\documentclass[12pt]{amsart}
\usepackage{enumerate}
\usepackage{amsthm}
\usepackage{amsmath,amssymb,latexsym,amscd}
\usepackage{tikz}
\makeatletter
\@namedef{subjclassname@2010}{%
  \textup{2010} Mathematics Subject Classification}
\makeatother
\theoremstyle{definition}
\newtheorem{theorem}{Theorem}[section]

\newtheorem{proposition}[theorem]{Proposition}

\theoremstyle{definition}
\theoremstyle{definition}

\usepackage{indentfirst}

\DeclareMathOperator{\im}{im}
\def\Z{\mathbb{Z}}

\def\S{\mathbb{S}}

\begin{document}
\baselineskip=17pt
\title[]{Finitistic Spaces with the orbit space $\mathbb{FP}^n\times \mathbb{S}^m$}
\author[Anju Kumari and Hemant Kumar Singh]{ Anju Kumari and Hemant Kumar Singh}
\address{\textbf{Anju Kumari}, Department of Mathematics, University of Delhi,	Delhi -- 110007, India.}
\email{anjukumari0702@gmail.com}
\address{\textbf{Hemant Kumar Singh}, 
	Department of Mathematics, University of Delhi 
	Delhi -- 110007, India.}
\email{hemantksingh@maths.du.ac.in}

\date{}
\thanks{This paper is supported by the Science and Engineering Research Board (Department of Science and Technology, Government of India) with reference number- EMR/2017/002192}
\begin{abstract} 
Let $G=\mathbb{S}^d$, $d=0,1$ or $3$,  act freely on a finitistic connected space $X$. This  paper gives 
the cohomology classification of $X$ if a mod 2 or rational cohomology of the orbit space $X/G$ is isomorphic to the  product of a projective space and  sphere $\mathbb{FP}^n\times \S^m$, where $\mathbb{F}=\mathbb{R},\mathbb{C}$ or $\mathbb{H}$, respectively.  For a free involution on $X$, a lower bound of  covering dimension of the coincidence  set of a continuous map $f:X\to \mathbb{R}^k$ is also determined.

\end{abstract}


\maketitle
\section {Introduction}
Let $G$ be a  group acting  on a compact Hausdorff  space $X$. It has been an interesting problem to classify the orbit space $X/G$ for a given space  $X$ when $G$ acts freely on $X$. It is well known that the projective spaces $\mathbb{FP}^n$, where $\mathbb{F}=\mathbb{R}$, $\mathbb{C}$ or $\mathbb{H}$,  are the orbit spaces of the standard  free actions of $G=\mathbb{S}^d$ on  $\mathbb{S}^{(d+1)n+d}$, where $d=0,1$ and $3$, and the lens space $L_p^{2n+1}$ is the orbit space of the standard free action of $G=\mathbb{Z}_p$, $p$ an odd prime, on $\mathbb{S}^{2n+1}$. 
A lot of work have been done in the literature in this direction for various spaces $X$ and different groups $G$. For example: (i) $X=L_p^{2n+1}$ or $\mathbb{FP}^n$ or $\mathbb{FP}^n\times \mathbb{FP}^m$, where $\mathbb{F}=\mathbb{R}$ or $\mathbb{C}$  and $G=\Z_2$
\cite{Hemant2008, Mahender2010, Mahender2013}, (ii) $X=\mathbb{F}H_{r,s}$, $\mathbb{F}=\mathbb{R}$ or $\mathbb{C}$ and $0\leq s\leq r$, (the Real or complex Milnor manifold) and $G=\Z_2$ or $\mathbb{S}^1$ \cite{Pinka Dey}, (iii) $X=\S^n\times\S^m, 1\leq n\leq m$ and $G=\Z_p, p$ a prime, or $\S^1$ or $\S^3$ \cite{Dotzel, Anju1}, and (iv) $X=P(1,n)$ (Dold manifold) and $G=\Z_2$ \cite{Morita}. 
  Another thread of research  is to classify a space $X$ for a given orbit space $X/G$, where $G$ acts   on  $X$. For semi-free actions of $G=\S^1$ with the orbit space $X/G$ is  a   
  simply connected 4-manifold, it has been proved that the total space is a  simply connected 5-manifold \cite{Harvey}.  Duan \cite{Duan 2005} shows that every 1-connected 4-manifold $M$ is the quotient of some  regular circle action on the space ${\#}_{r-1}\mathbb{S}^2\times \mathbb{S}^3$, where $r=\text{rank }H^2(M,\Z)$. If $G$ acts freely on $X$, then the cohomology classification of $X$ is discussed for (i) $X/G=L_p^{2n+1}$ or $\mathbb{CP}^n$ and $G=\Z_p, p$ a prime or $\mathbb{S}^1$ \cite{Su1963}, and for $X/G=\mathbb{HP}^n$ and $G=\S^3$ \cite{Anju 2}.   We contribute to this question by determining  the  cohomology classification of a finitistic space $X$ equipped with  free action of  $G=\Z_2,\mathbb{S}^1$ or $\mathbb{S}^3$ and the 
  orbit space  $X/G$ is a mod $2$ or rational cohomology  $\mathbb{FP}^n\times \mathbb{S}^m,$ where $\mathbb{F}=\mathbb{R},\mathbb{C}$ or $\mathbb{H}$, respectively.  These results describes   the converse of some results for $G=\Z_2$ or $\S^1$ actions proved by Dotzel et. al. \cite{Dotzel}, and  for $G=\mathbb{S}^3$ actions proved in \cite{Anju1}.   
  
  For free $\Z_2$-space $X$ and any space  $Y$, the   coincidence set of a continuous map  $f:X\to Y$  is defined as 
  $A(f)=\{x\in X|f(x)=f(gx) \text{ for each } g\in \Z_2\}$.
  The classical Borsuk-Ulam Theorem states that for  a continuous map $f:\mathbb{S}^n\to \mathbb{R}^k$, the coincidence set $A(f)=\{x\in \mathbb{S}^n|f(x)=f(-x)\}$ is nonempty, where $\S^n$ is equipped with antipodal action. 
  Munkholm \cite{Munkholm} shows that for a map $f:\S^n\to M$, topological dimension of coincidence set $A(f)$ is greater than or equal to $n-k$, where  $M$ is a compact $k$-dimensional topological manifold, $n>k$. Biasi et al. \cite{Biasi} extends this result for generalized manifolds. In this paper, we determine a  lower bound of the covering dimension of  coincidence set $A(f)$ of a continuous map  $f:X\to \mathbb{R}^k$  if $H^*(X/G;\Z_2)=\Z_2[a,b]/\langle a^{n+1},b^2\rangle$, where $\deg a= 1$ and $\deg b=m$.

\section{Preliminaries}  
Let $G$ be a compact Lie group and $G\to E_G\to B_G$ be the universal principal $G$-bundle, where $B_G$ is the classifying space. Suppose $G$  acts freely on a finitistic space $X$. The associated bundle $X\hookrightarrow (X\times E_G)/{G}\to B_G$ is a fibre bundle with fibre $X$. Put $X_G=(X\times E_G)/{G}$. The bundle $X\hookrightarrow X_G\to B_G$ is called the Borel fibration. Then there exists the Leray-Serre spectral sequence for the Borel fibration  $X\stackrel{i}{\hookrightarrow} X_G \stackrel{\pi}{\rightarrow} B_G$ which converges to $H^*(X_G)$ as an algebra with $E_2^{k,l}=H^k(B_G;\mathcal{H}^l(X;R))$ . If $B_G$ is simply connected then the system of local coefficients on $B_G$ is simple then the $E_2$-term  becomes
\begin{equation*}
	E_2^{k,l}= H^k(B_G;R)\otimes H^l(X;R).
\end{equation*}
We recall some results which are needed to prove our results:
\begin{proposition}\label{5p2}
	Let $X\stackrel{i}{\hookrightarrow} X_G \stackrel{\pi}{\rightarrow} B_G$ be the Borel fibration. Suppose that the system of local coefficients on $B_G$ is simple, then the edge homomorphisms
	\begin{center}
			$ H^k(B_G)\cong E_2^{k,0} \longrightarrow E_3^{k,0}\longrightarrow \cdots 		
			  \longrightarrow E_k^{k,0} \longrightarrow E_{k+1}^{k,0}=E_{\infty}^{k,0}\subset H^k(X_G), \text{ and }$ \\
$  H^l(X_G) \longrightarrow E_{\infty}^{0,l}= E_{l+1}^{0,l} \subset E_{l}^{0,l} \subset \cdots \subset E_2^{0,l}\cong  H^l(X) $
	\end{center}
	are the homomorphisms $$\pi^*: H^k(B_G) \to H^k(X_G) ~ ~ ~ \textrm{and} ~ ~ ~ i^*: H^l(X_G)  \to H^l(X).$$
\end{proposition}
For details about spectral sequences, we refer the reader to \cite{McCleary}. Let $h:X_G\rightarrow X/G$ be the map induced by the $G$-equivariant projection $X\times E_G\rightarrow X$.  Then h is a homotopy equivalence \cite{Dieck}.   All the cohomologies are  \v{C}ech cohomology  with coefficients in $R$, where $R=\mathbb{Q}$ or $\mathbb{Z}_2$. Note that $X\sim_R Y$ means $H^*(X;R)\cong H^*(Y;R)$. For $R=\mathbb{Q}$ and for $G=\mathbb{S}^1$ or $\mathbb{S}^3$, we assume that the associated sphere bundle $G\hookrightarrow X\to X/G$ is orientable.
\begin{proposition}(\cite{Hemant2012,JKaur2015})\label{S1 and S3 pre}
	Let $G=\S^1$ or $\S^3$, act freely on a  finitistic space $X$.  If $H^i(X;R)=0$ for all $i>n$, then $H^i(X/G;R)=0$ for all $i> n$.
\end{proposition}
Now, we recall Gysin sequence of sphere bundles.
\begin{proposition}
	Let  $G=\mathbb{S}^d$, $d=0,1$ or $3$, act freely on a finitistic space $X$. The Gysin sequence of the sphere bundle $G\hookrightarrow X\stackrel{p}{\to} X/G$ is: \begin{align*}
		\cdots\rightarrow H^{i}(X/G) \stackrel{p^*_{i}}{\longrightarrow}H^{i}(X)\stackrel{\rho_{i}}{\longrightarrow}H^{i-d}(X/G)\stackrel{\cup}{\longrightarrow} H^{i+1}(X/G)\stackrel{p_{i+1}^*}{\longrightarrow}H^{i+1}(X)
		\rightarrow\cdots
	\end{align*}
	which  start with 
	\begin{align*}
		0\longrightarrow &H^{d}(X/G)\stackrel{p^*_d}{\longrightarrow}H^{d}(X)\stackrel{\rho_{d}}{\longrightarrow}H^0(X/G)\stackrel{\cup}{\longrightarrow}H^{d+1}(X/G)\stackrel{p^*_{d+1}}{\longrightarrow}H^{d+1}(X)\longrightarrow\cdots 	\end{align*}
 	where $\cup:H^i(X/G)\to H^{i+d+1}(X/G)$  maps $x\to x\cup u$ and $u\in H^{d+1}(X/G)$ denotes the Euler class of the  sphere  bundle.  It is easy to observe that $p^*_i$ is an isomorphism for $0\leq i\leq d-1$ for $d=1$ or 2. 
\end{proposition}
In this paper, we have considered finitistic spaces. These spaces were introduced by R. G. Swan \cite{Swan} and have been discovered as relevant spaces for the study of cohomological aspects of tranformation groups \cite{bredon}. 
Recall that a space  is said to be finitistic if it is paracompact, Hausdorff and every open cover of it have a finite dimensional open refinement. All compact spaces, all finite dimensional and
all finite-dimensional paracompact spaces are some examples of  finitistic spaces. Note that the space $X=\prod\limits_{n=1}^{\infty} \mathbb{S}^n\times\mathbb{R}^k$ is an example of  finitistic space which is neither compact nor has finite covering dimension. Recall that $H^*(\mathbb{RP}^n\times \mathbb{S}^m;R)=R[a,b]/\langle a^{n+1},b^2  \rangle$, where $\deg a=1$ and $\deg b=m$.

\section{Main Theorems} 

Let $G=\mathbb{S}^d$, $d=0,1$ or $3$, act freely on a  finitistic space $X$ having the mod $2$ or rational cohomology the  product of spheres $\mathbb{S}^{(d+1)n+d}\times \mathbb{S}^m$. It has been proved that   one of the possibilities  of the orbit space $X/G$  are the mod $2$ or rational cohomology  $\mathbb{FP}^n\times\mathbb{S}^m$, where $\mathbb{F}=\mathbb{R},\mathbb{C}$ or $\mathbb{H}$, respectively.                                             
Using techniques of the Gysin sequence of sphere bundles and the Leray-Serre spectral sequence of Borel fibration, 
 we discuss converse of these statements. These results also describes the cohomology classification of a finitistic connected free $G$-space $X$  with the orbit space a product of projective space  and sphere. First, we have discussed  for free actions of $G=\S^3$.        
\begin{theorem}\label{Thoerem S3}
	Let $G=\mathbb{S}^3$ act freely on a finitistic connected space $X$  with $X/G\sim_R\mathbb{HP}^n\times \mathbb{S}^m$, where $R$ is $\mathbb{Q}$ or $\Z_2$. Then the cohomology algebra of  $X$ with coefficients in $R$ is isomorphic to the cohomology algebra of one of the following: 
	\begin{enumerate}
		\item[(i)] $\mathbb{S}^m\times\mathbb{S}^{4n+3}$; 
		\item[(ii)]  $\mathbb{S}^3\times\mathbb{S}^{m}\times\mathbb{HP}^n$;
		
		\item[(iii)] $\mathbb{HP}^n\times\mathbb{S}^7$ and $m=4$;
		\item[(iv)]$R[x,y]/ \langle x^{n+1},y^4 \rangle$,  where $\deg x=4$, $\deg y=3$,  $m=6$,   $\beta(y)=x$ and $R=\Z_2$,  where $\beta:H^3(X/G)\to H^4(X/G)$ denotes the  Bockstein homomorphism  associated with the coefficient sequence $0\to \Z_2\to Z_4\to \Z_2\to 0$.
		
	\end{enumerate}

\end{theorem}   
\begin{proof}
It is clear that $H^i(X/G)\cong H^i(X)$ for  $i=0,1$ and 2, and $H^i(X)=0$ for all $i>m+4n+3$.  We consider  different possibilities  of the Euler class in the Gysin sequence:

\noindent\textbf{Case(a):} When $\cup:H^0(X/G)\to H^4(X/G)$ is trivial. 

First, assume that $m>4n$.
In this case, $\cup:H^{i}(X/G)\to H^{i+4}(X/G)$ is trivial for all $ i\geq 0$.  For $0\leq i< n$ and $k=0,m$, we get $\rho_{k+4i+3}$ and $p^*_{k+4i+4}$ are isomorphisms.   This implies that  $H^{k+4i+3}(X)\cong H^{k+4i+4}(X)\cong R$. Suppose that $\{u_{4i+3}\}$ and $\{x^{i+1}\}$ denotes bases for   $H^{4i+3}(X)$ and $H^{4i+4}(X)$,   respectively, where $\rho_{4i+3}(u_{4i+3})=a^i$ and $p^*_{4}(a)=x$. Note that  $H^i(X)=0$ for   $i\not = 4n+3$, and $4n<i<m$.  If $m\not =4n+3$, then by the exactness of the Gysin sequence,  $H^{4n+3}(X)\cong H^m(X)\cong R$   with bases $\{u_{4n+3}\}$ and $\{y\}$,  respectively; and if $m=4n+3$ then $H^{4n+3}(X)\cong R\oplus R$ with basis $\{u_{4n+3},y\}$, where $\rho_{4n+3}(u_{4n+3})=a^n$ and  $p^*_{m}(b)=y$.  Also, $H^{m+4n+3}(X)\cong R$. Let  $\{v_{m+4i+3}\}$ be basis for $H^{m+4i+3}(X)$ with $\rho_{m+4i+3}(v_{m+4i+3})=a^ib$ for all $0\leq i\leq n$.    Clearly, $H^{j}(X)=0$ for all $k\leq j\equiv k+1\text{ or }k+2 (\text{mod } 4)<k+4n$; and $k=0,m$; and  $j\not=4n+3$. Note that $u_3^2=0$.

Now we assume that $m\leq 4n$. We consider four cases $m\equiv j(\text{mod }4),0\leq j\leq 3$.\\
\textbf{m\,$\equiv$\,0(mod 4):} Let
  $m=4i_0$ for some $i_0\leq n$. In this case, $H^i(X)=0$ for $i\equiv 1 (\text{mod } 4)$ or $i\equiv 2 (\text{mod } 4) $.  Also, $\rho_{4i-1}$ and $p^*_{4i}$  are isomorphisms for all $i\geq 0$. This implies that for each $0\leq i<i_0$, we have $H^{4i}(X)\cong R$  and $H^{4i+3}(X)\cong R$ with bases $\{x^i\}$ and $\{u_{4i+3}\}$, respectively, and for $i_0\leq i\leq n$, $H^{4i}(X)\cong R\oplus R$ and $H^{4i+3}(X)\cong R\oplus R$ with bases $\{x^i,x^{i-i_0}y\}$ and $\{u_{4i+3},v_{4i+3}\}$, respectively, where $p^*_{m}(b)=y,p^*_{4}(a)=x,\rho_{4i+3}(u_{4i+3})=a^i$ and  $\rho_{4i+3}(v_{4i+3})=a^{i-i_0}b$. Also, for $1\leq i\leq i_0$, $H^{4n+4i}(X)\cong R$ and $H^{4n+4i+3}(X)\cong R$ with bases $\{x^{n+i-i_0}y\}$ and $\{v_{4n+4i+3}\}$, respectively, where $\rho_{4n+4i+3}(v_{4n+4i+3})=a^{n+i-i_0}b$. Clearly,  $u_3^2=0$. Thus, 
$$H^i(X)=\begin{cases}
	R &\text{ if }j\leq i\equiv 0\text{ or }3 (\text{mod } 4)\leq m-1+j, j=0\text{ or } 4n+4 \\
	R\oplus R &\text{ if }m\leq i\equiv 0 \text{ or }3(\text{mod } 4)\leq 4n+m+3\\
	0 &\text{ otherwise.}
\end{cases}$$
\textbf{m\,$\equiv$\,1(mod 4):} Let  $m=4{i_0}+1$ for some $i_0< n$. For  $i\not=i_0+j,1\leq j\leq n+1$, we get
  $p^*_{4i}$ are isomorphisms. This implies that   $H^{4i}(X)\cong R$ with basis $\{x^i\}$ for $0\leq i\leq i_0$, where $p^*_{4}(a)=x$. For $i_0+1\leq i\leq n$, $H^{4i}(X)\cong R\oplus R$ with basis $\{x^i,v_{4i}\}$; and for $n+1\leq i\leq i_0+n+1$, $H^{4i}(X)\cong R$ with basis $\{v_{4i}\}$, where $\rho_{4i}(v_{4i})=a^{i-i_0-1}b$. It is clear that  $H^{4i+2}(X)=0$ for all $i\geq 0$; $H^{4i+3}(X)\cong R\cong H^{m+4i}(X)$, for all $0\leq i\leq n$, with bases $\{u_{4i+3}\}$ and $\{x^iy\}$, where $\rho_{4i+3}(u_{4i+3})=a^i$ and $p^*_{m}(b)=y$; and  $H^{4i+3}(X)=H^{4i+1}(X)=0$ otherwise. Obviously, $u_3^2=0$. We have $$H^i(X)=\begin{cases}
	R &\text{ if }j\leq i\equiv 0(\text{mod }4)\leq j+m-1,j=0\text{ or }4n+4,\text{ or }\\
	&\;\;\;\;     j\leq i\equiv j(\text{mod }4)\leq j+4n,j=3\text{ or } m\\ 
	R\oplus R &\text{ if }m< i\equiv 0(\text{mod }4)\leq 4n\\
	0&\text{ otherwise.} \end{cases}$$  
\textbf{m\,$\equiv$\,2(mod 4):} Let $m=4i_0+2$ for some $i_0<n$. We get $\rho_{4i+3}$ and $p^*_{4i}$ are isomorphisms, for all $i\geq 0$. Therefore, for $0\leq i\leq n$, $H^{4i+3}(X)\cong H^{4i}(X)\cong R$ with bases $\{u_{4i+3}\}$ and $\{x^i\}$, respectively,   and for $i>n$, $H^{4i+3}(X)=H^{4i}(X)=0$, where $\rho_{4i+3}(u_{4i+3})=a^i$ and $p^*_4(a)=x$. Also, we have $H^{4i+1}(X)=0$, for $i\not=i_0+j,1\leq j\leq n+1$,  and for $i_0+1\leq i\leq i_0+n+1$, $H^{4i+1}(X)\cong H^{4i-2}(X)\cong R$ with bases $\{v_{4i+1}\}$ and $\{x^{i-i_0-1}y\}$, where $\rho_{4i+1}(v_{4i+1})=a^{i-i_0-1}b$ and $p^*_m(b)=y$. 
Consequently,
$$H^i(X)=\begin{cases}
	R &\text{ if }j\leq i\equiv j\text{ or }j+3(\text{mod }4)\leq j+4n+3, j=0\text{ or } m\\
		0&\text{ otherwise.}
\end{cases}$$
\textbf{m\,$\equiv$\,3(mod 4):} Let $m=4i_0+3$ for some $i_0< n$. Then $H^{4i}(X)=0$ for $i>n,$ and $H^{4i}(X)\cong R$ with  basis $\{x^i\},$ for $0\leq i\leq n$, where $p^*_4(a)=x$. Also,  $H^{4i+1}(X)=0$ for all $i\geq 0$.  By the exactness of the Gysin sequence, for $i_0+1\leq i\leq i_0+n+1$,  $H^{4i+2}(X)\cong R$ with basis $\{v_{4i+2}\}$,  where  $\rho_{4i+2}(v_{4i+2})=a^{i-i_0-1}b$, and  $H^{4i+2}(X)=0$,  otherwise. Now, for $0\leq i< i_0$, $H^{4i+3}(X)\cong R$ with basis $\{u_{4i+3}\}$; for $i_0\leq i\leq n$, $H^{4i+3}(X)\cong R\oplus R$ with basis $\{u_{4i+3},x^{i-i_0}y\}$;  and for $n<i\leq i+i_0$, $H^{4i+3}(X)\cong R$ with basis $\{x^{i-i_0}y\}$, where $\rho_{4i+3}(u_{4i+3})=a^i$ and $p^*_m(b)=y$. Clearly, $u_3^2=0$.  We have
$$H^i(X)=\begin{cases}
	R&\text{ if }j\leq i\equiv j(\text{mod }4)\leq j+4n, j=0\text{ or } m+3,\text{ or }\\
	& \;\;\;\;j+3\leq i\equiv 3(\text{mod }4)\leq j+m-4,j=0 \text{ or } 4n+4\\
	R\oplus R&\text{ if }m\leq i\equiv 3(\text{mod }4)\leq 4n+3\\
	0&\text{ otherwise.} 
\end{cases}$$

Consider  the Leray-Serre spectral sequence for the Borel fibration $X\stackrel{i}{\hookrightarrow}X_G\stackrel{\pi}{\rightarrow}B_G$ which converges to $H^*(X_G)$ as an algebra and 	$E_2^{k,l}= H^k(B_G;R)\otimes H^l(X;R)$ for all $k,l\geq 0$.
The possible nontrivial differentials in the Leray-Serre spectral sequence for the Borel fibration $X\stackrel{i}{\hookrightarrow}X_G\stackrel{\pi}{\rightarrow}B_G$  are $\{d_{4r}\}_{r\geq 1}$.  By the edge homomorphisms and the fact that $p^*_j=i^*\circ h^*$ for all $j\geq 0$, we get $d_r(1\otimes x)=0=d_r(1\otimes y)$ for all $r\geq 0$. Note that  $u_{4i+3}$ and $v_{m+4i+3}$ are not in image of $p^*_{4i+3}$ and $p^*_{m+4i+3}$, respectively, so their image must be nonzero under some differential for all $0\leq i\leq n$. Consequently,  $x^iu_3\not=0$ and $x^iu_3y\not=0$, for all $i$. Obviously, $x^{n+1}=y^2=0$.

For $m>4n$, it is clear that   $u_{4i+3}=\alpha_i x^iu_3$ and $v_{m+4i+3}=\beta_i x^iu_3y$ for some nonzero elements $\alpha_i,\beta_i$ in $R$ and  $0\leq i\leq n$. In particular, for $m=4n+3$, $x^nu_3$ can not be equal  to any multiple  of $y$ and so $u_{4n+3}$ is  generated by $\{y,x^nu_3\}$.

Now, suppose $m\leq 4n$. If $m=4i_0$, for some $i_0\leq n$ then $u_{4i+3}=\alpha_i x^iu_3$ for $0\leq i<i_0$ and  $v_{4i+3+m}=\beta_i x^iu_3y$  for $n+1-i_0\leq i\leq n$, where  $\alpha_i$'s  and $\beta_i$'s are nonzero elements in $R$. We observe that 
 $x^iu_3$ can not be equal to any multiple of $x^{i-i_0}u_3y$,  for all $ i_0\leq i\leq n$. Therefore,  the elements  $u_{4i+3}$ and $v_{4i+3}$ are generated by $\{ x^iu_3, x^{i-i_0}u_3y \}$ for each $i_0\leq i\leq n$.
If $m=4i_0+1 $ for some $i_0<n$ then  the elements $x^iu_3y$ cannot be equal to any multiple of $x^{4i+3+m}$ for each $0\leq i\leq n-i_0+1$. Consequently, $v_{4i+m+3}$ is generated by $\{x^iu_3y,x^{4i+m+3}\}$  and $u_{4i+3}$ is generated by $\{x^iu_3\}$ for all $0\leq i\leq n$.   If $m=4i_0+2$ for some $i_0<n$ then we have $u_{4i+3}=\alpha_i x^iu_3$ and $v_{m+4i+3}=\beta_i x^iu_3y$ for some nonzero elements $\alpha_i,\beta_i$ in $R$ and  $0\leq i\leq n$. In this case,  $u_3^2$ may be both zero or nonzero.  If  $m=4i_0+ 3$ then  $v_{m+4i+3}=\alpha_i x^iu_3y$,  where $\alpha_i\in R$, for all $0\leq i\leq n$. As $x^iu_3$ can not be equal to $x^{i-i_0}y$, for all $i_0\leq i\leq n$, therefore, $u_{4i+3}$ is generated by $\{x^iu_3,x^{i-i_0}y\}$, for all $0\leq i\leq n$. For $m\not=3$, obviously $u_3^2=0$. If $u_3^2\not=0$, for $m=3$, then 
  $d_4(1\otimes v_{6})=0$, a contradiction. 
    Therefore, $u_3^2=0$. The cohomology algebra of $X$ is given by 
$$H^*(X)=R[x,u_3,y]/\langle x^{n+1},u_3^2,y^2\rangle,$$ where $\deg x=4$, $\deg u_3=3$ and $\deg y=m$.  This realizes the  possibility (ii).

If $u_3^2\not =0$ then $m$ must be 6  and $u^2_3=\alpha y$ for some $\alpha$ nonzero in $R$. By the commutativity of cup product, we get $2u_3^2=0$. This implies that  $R$ cannot be $\mathbb{Q}$. Therefore, 
$$H^*(X)\cong R[x,u_3] /\langle x^{n+1},u_3^4 \rangle,$$ $\deg x=4$ and $\deg u_3=3$ only when $R=\Z_2$. By the properties of Steenrod squares $Sq^3(y)=y^2\not=0$. As $3$ is not a power of 2, we get $Sq^3=Sq^2\circ Sq^1+Sq^1\circ Sq^2$. This gives that $Sq^1(y)=x$. Note that $Sq^1$ is the Bockstein homomorphism  associated with the coefficient sequence $0\to \Z_2\to Z_4\to \Z_2\to 0$.  
  This realizes the  possibility  (iv) of the Theorem.

\noindent
\textbf{Case(b):} When $\cup:H^0(X/G)\to H^4(X/G)$ maps $1$ to $ca$ for some $c\not=0$ in $R$. 

First, we suppose that $m>4n$.  For $0\leq i<n$ and $j=0,m$, we get $\cup:H^{4i+j}(X/G)\to H^{4i+4+j}(X/G)$  are  isomorphisms; $\rho_{4i+3+j}$ and  $p^*_{4i+4+j}$ are trivial homomorphisms. Consequently,  $H^{4i+3+j}(X)=H^{4i+4+j}(X)=0$. Also, $H^{4i+1+j}(X)=H^{4i+2+j}(X)=0$ except for $H^{4n+3}(X)$. Note that $H^i(X)=0$ if $(4n<i<m \text{ and }i\not=4n+3)$  and $H^{m+4n+3}(X)\cong R$ with basis $\{u_{m+4n+3}\}$, where $\rho_{m+4n+3}(u_{m+4n+3})=a^nb$.
 For $m\not = 4n+3$,  $\rho_{4n+3}$ and $p^*_m$ are isomorphisms. Consequently, $H^{4n+3}(X)\cong R\cong H^m(X)$ with bases $\{u_{4n+3}\}$ and $\{y\}$, respectively, where $p^*_m(b)=y$ and $\rho_{4n+3}(u_{4n+3})=a^n$. If $m=4n+3$ then $H^{4n+3}(X)\cong R\oplus R$ with basis $\{u_{4n+3},y\}$, where $\rho_{4n+3}(u_{4n+3})=a^n$ and $p^*_m(b)=y$.  Also, $H^{m+4n+1}(X)=H^{m+4n+2}(X)=0$. For  $m\not=4n+3$, we get 
$$H^i(X)=\begin{cases}
	R&\text{ if }i=0,4n+3,m,m+4n+3\\
	0&\text{ otherwise}
\end{cases}$$  
  and for $m= 4n+3$, we get
$$H^i(X)=\begin{cases}
	R\oplus R&\text{ if }i=4n+3\\
	R&\text{ if }i=0,m+4n+3\\
	0&\text{ otherwise.}
\end{cases}$$ 
For $m\leq 4n$, we get the same cohomology groups as in $m>4n$.

Now, we compute the cohomology ring structure of $X$.
In the Leray-Serre spectral sequence for the Borel fibration $X\hookrightarrow X_G\stackrel{\pi}{\to }B_G$, $d_{4r'}(1\otimes u_{4n+3})\not=0$  for some $r'>0$ and $d_{4r}(1\otimes y)=0$ for all $r\geq 0$. Firstly, suppose that $m=4i_0$   and $r'=n-i_0+1$, where $1\leq i_0\leq n$, then 
$E_{4n-m+5}^{k,4n+3}=0,E_{4n-m+5}^{k,q}=E_2^{k,q}$ for  all $k\geq 0$ and $q=0,m+4n+3$. Also, $E_{4n-m+5}^{k,m}=E_2^{k,m}$ for all $k\leq 4n-m$ and trivial otherwise. Since $G$ acts freely on $X$, therefore, $d_{m+4n+4}(1\otimes u_{m+4n+3})=ct^{n+i_0+1}\otimes 1$ for some $c\not=0$ in $R$.   Then  $E_{\infty}^{4k,q}=R$ for all ($4k\leq m+4n,q=0$), ($4k\leq 4n-m, q=m$) and trivial otherwise.  This implies that $t\otimes 1\in E_2^{4,0}$ and $1\otimes y \in E_2^{0,m}$ are permanent cocycles. Then by edge homomorphism  there exist  $u\in E_{\infty}^{4,0}$ and  
 $w\in E_{\infty}^{0,m}$  corresponding  $t\otimes 1$ and   $1\otimes y$ respectively with $\pi^*(t)=u$. We have  $w^2=u^{n+i_0+1}=u^{n-i_0+1}w=0$. This implies that 
\begin{align*}
	\textnormal{Tot}E_{\infty}^{*,*}\cong \mathbb{R}[u,w]/\langle w^2,u^{n+i_0+1},u^{n-i_0+1}w \rangle, 
\end{align*}
	where $\deg u=4\text{ and } \deg w=m.$ Then there exist an element $v\in H^m(X_G)$ corresponding to $w\in E_{\infty}^{0,m}$ such that $i^*(v)=y$. We have  $u^{n-i_0+1}v=\alpha u^{n+1}$  and $v^2=\beta u^{2i_0}+\gamma u^{i_0}v$, where $\alpha,\beta,\gamma\in R$ and $\beta=0$  if  $m>4n$. So, the ring cohomology of $X/G$ is given by $H^*(X_G)\cong R[u,v]/\langle u^{n-i_0+1}v-\alpha u^{n+1},v^2-\beta u^{2i_0}-\gamma u^{i_0}v ,u^{n+i_0+1}\rangle$
where $\deg u=4,\deg v=m$ and $\alpha,\beta,\gamma\in R$, 
$\beta=0$  if  $m>4n$, which is a contradiction. So,  $r'$ must be $n+1$. This implies that  $d_{4n+4}(1\otimes yu_{4n+3})\not = 0$. Consequently, $u_{4n+m+3}=\alpha yu_{4n+3}$ for some $\alpha \not=0$ in $R$. Obviously, $y^2=0$, and $u_{4n+3}^2=0$ for $m\not\in\{4n+3,8n+6\}$.
If $m=4n+3$ then  $u_{4n+3}^2\not =\beta u_{4n+m+3}$ for any $\beta $ in $R$ and we get $u_{4n+3}^2=0$. If $m=8n+6$ then $u_{4n+3}^2$ may be both zero or nonzero. Thus, $$H^*(X,R)=R[y,u_{4n+3}]/\langle y^2,u_{4n+3}^2\rangle,$$ where $\deg y=m$ and $\deg u_{4n+3}=4n+3$. This realizes possibility (i).  If $u_{4n+3}^2\not=0$ then $u_{4n+3}=\alpha y $ for some nonzero element $\alpha\in R$. Thus, the cohomology algebra of $X$ is given by 
$R[u_{4n+3}]/\langle u_{4n+3}^4\rangle,$ where $\deg u_{4n+3}=4n+3$. By \cite[Theorem 4L.9]{Hatcher}, this cohomology algebra is not possible for  $R=\Z_2$ and by the commutativity of cup product this cohomology algebra is also not possible for $R=\mathbb{Q}$.

  \noindent\textbf{Case(c):} When $\cup:H^0(X/G)\to H^4(X/G)$ maps $1$ to $cb$, where $c\not=0$ in $R$.

   In this case $m$ must be 4, and we get $H^{4i}(X)\cong R$ with basis $\{x^i\}$, where $p^*_4(a)=x$. Clearly,  $H^3(X)=H^{4i+1}(X)=H^{4i+2}(X)=0$ for all $i\geq 0$. As $\ker\cup:H^{4i}(X/G)\to H^{4i+4}(X/G)$ is generated by $a^{i-1}b$, we have $H^{4i+3}(X)\cong R$ with basis $\{u_{4i+3}\}$, where $\rho_{4i+3}(u_{4i+3})=a^{i-1}b$ for all $1\leq i\leq n+1$. We must have $d_4(1\otimes u_7)=0$ and $d_8(1\otimes u_7)\not=0$. Consequently,  $u_{4i+7}=c_ix^ib_7$ for some nonzero $c_i\in R$,  for all $1\leq i\leq n$. Therefore, $H^*(X)\cong R[x,u_7]/\langle u_7^2, x^{n+1}\rangle$, where $\deg x=4$ and $\deg u_7=7$.
This realizes the possibility (iii). 

\noindent\textbf{Case(d):}  When $\cup:H^0(X/G)\to H^4(X/G)$ maps $1$ to $ca+c'b$, where $c,c'\not=0$ in $R$.

In this case also $m$ must be 4 and $H^{4}(X)\cong R$ with basis $\{x\}$, where $p^*_4(a)=x$. By the exactness of the Gysin sequence, we get $H^3(X)=H^{4i+1}(X)=H^{4i+2}(X)=0$ for all $i\geq 0$. As $\cup:H^{4i}(X/G)\to H^{4i+4}(X/G)$ is an isomorphism, we get $H^{4i+3}(X)=H^{4i+4}(X)=0$ for all $0< i<n$. Note that $H^{4n+4}(X)=0$ and $\ker\cup:H^{4n}(X/G)\to H^{4n+4}(X/G)$ is generated by $\{a^{n-1}b-\frac{c}{c'}a^n\}$. This implies that  $H^{4n+3}(X)\cong R$ with basis $\{u_{4n+3}\}$, where $\rho_{4n+3}(u_{4n+3})=a^{n-1}b-\frac{c}{c'}a^n$.   
Obviously, $H^{4n+7}(X)\cong R$ with basis $\{u_{4n+7}\}$, where $\rho_{4n+7}(u_{4n+7})=a^nb$.   
Using similar arguments as in case(b), $H^*(X)\cong R[x,u_{4n+3}]/\langle u_{4n+3}^2, x^{2}\rangle$, where $\deg x=4$ and $\deg u_{4n+3}=4n+3$. This  realizes the possibility (i). 
\end{proof}
 
Next, we discuss similar  result for circle actions with the orbit space  product of a  complex projective space and  sphere:

\begin{theorem}\label{Theorem for S^1}
	Let $G=\mathbb{S}^1$ act freely on a finitistic connected space $X$  with $X/G\sim_R\mathbb{CP}^n\times \mathbb{S}^m$,  where $R$ is $\mathbb{Q}$ or $\Z_2$.  The cohomology algebra of  $X$ with coefficients in $R$ is isomorphic to the cohomology algebra of one of the following:
	\begin{enumerate}
		\item[(i)] $\mathbb{S}^m\times\mathbb{S}^{2n+1}$; 
		\item[(ii)] $\mathbb{S}^1\times\mathbb{S}^{m}\times\mathbb{CP}^n;$
		\item[(iii)] $\mathbb{RP}^{2n+1}\times\mathbb{S}^m;$ 
		\item[(iv)] $R[x,y]/\langle x^{n+1},y^2+\alpha x^3 \rangle$, where $\deg x=2,\deg y=3,m=2$, $\alpha\in R$ and $\alpha=0$ for $R=\mathbb{Q}$.
	\end{enumerate}  
\begin{proof}
	As $X$ is connected,  $H^0(X/G)\cong H^0(X)$, and clearly, $H^i(X)=0$ for all $i>m+2n+1$.  We consider the following cases:

	\noindent\textbf{Case(a):} When $\cup:H^0(X/G)\to H^2(X/G)$ is trivial. \\
First assume that $m>2n$.	In this case, $\cup:H^i(X/G)\to H^{i+2}(X/G)$ is trivial for all $i\geq 0$.  Then $\rho_{k+2i+1}$ and $p^*_{k+2i+2}$ are isomorphisms for all $0\leq i< n$ and $k=0,m$.   This implies that  $H^{k+2i+1}(X)\cong H^{k+2i+2}(X)\cong R$. Suppose that $\{u_{2i+1}\}$ and $\{x^{i+1}\}$ denotes the bases for   $H^{2i+1}(X)$ and $H^{2i+2}(X)$,   respectively,  where $\rho_{2i+1}(u_{2i+1})=a^i$ and $p^*_2(a)=x$.  Also, $H^{m+2n+1}(X)\cong R$. Let $\{v_{m+2i+1}\}$ be basis  for $H^{m+2i+1}(X)$ with       $\rho_{m+2i+1}(v_{m+2i+1})=a^ib$  for all $0\leq i\leq n$. Note that  $H^i(X)=0$, for all   $2n+1<i<m$.  If $m\not =2n+1$, then by the exactness of Gysin sequence,  $H^{2n+1}(X)\cong H^m(X)\cong R$   with bases $\{u_{2n+1}\}$ and $\{y\}$,  respectively; and if $m=2n+1$ then $H^{2n+1}(X)\cong R\oplus R$ with basis $\{u_{2n+1},y\}$, where $\rho_{2n+1}(u_{2n+1})=a^n$ and  $p^*_m(b)=y$.

Now, suppose that $m\leq 2n$. We consider two cases when $m$ is even or odd:

\noindent \textbf{m is odd}: Let $m=2i_0+1$ for some $i_0< n$.  In this case, $\cup:H^{2i}(X/G)\to H^{2i+2}(X/G)$ is trivial for all $i\geq 0.$  Then $\rho_{2i-1}$ and  $p^*_{2i}$ are isomorphisms for  all $0\leq i\leq i_0$. This implies that   $H^{2i-1}(X)\cong  H^{2i}(X)\cong R$ with bases $\{u_{2i-1}\}$ and $\{x^i\}$, respectively, where $\rho_{2i-1}(u_{2i-1})=a^{i-1}$ and $p^*_2(a)=x$. For $i_0+1\leq i\leq n$, $H^{2i}(X)\cong H^{2i-1}(X)\cong H^{2n+1}(X)\cong   R\oplus R$ with bases $\{x^i,v_{2i}\}$,  $\{yx^{i-i_0-1},u_{2i-1}\}$ and  $\{yx^{n-i_0},u_{2n+1}\}$, respectively; and for $n<i\leq i_0+n$, we have $H^{2i}(X)\cong H^{2n+m+1}(X)\cong H^{2i+1}(X)\cong R$ with bases $\{v_{2i}\}$, $\{v_{2n+m+1}\}$ and $\{yx^{i-i_0}\}$, respectively, where $\rho_{2i}(v_{2i})=a^{i-i_0-1}b$, $\rho_{2i+1}(u_{2i+1})=a^i$ and $p^*_m(b)=y$.

	\textbf{m is even:} Let $m=2i_0$ for some $i_0\leq n$.  We get that  $\rho_{2i-1}$ and $p^*_{2i}$  are isomorphisms for all $i\geq 0$. This implies that  for each $0\leq i<i_0$, $H^{2i}(X)\cong H^{2i+1}(X)\cong  R$  with bases $\{x^i\}$ and $\{u_{2i+1}\}$, respectively,  and for $i_0\leq i\leq n$, we get $H^{2i}(X)\cong H^{2i+1}(X)\cong R\oplus R$  with bases $\{x^i,x^{i-i_0}y\}$ and $\{u_{2i+1},v_{2i+1}\}$, respectively, where $p^*_2(a)=x, p^*_m(b)=y,\rho_{2i+1}(u_{2i+1})=a^i$ and  $\rho_{2i+1}(v_{2i+1})=a^{i-i_0}b$. Also, for $1\leq i\leq i_0$, $H^{2n+2i}(X)\cong  H^{2n+2i+1}(X)\cong R$  with bases $\{x^{n+i-i_0}y\}$ and $\{v_{2n+2i+1}\}$, respectively, where $\rho_{2n+2i+1}(v_{2n+2i+1})=a^{n+i-i_0}b$.

	 Finally, for all $m\leq 2n$, we have 
		 $$H^i(X)=\begin{cases}
	 	R &\text{ if }j\leq i\leq m+j-1, j=0\text{ or }2n+2,\\
	 	R\oplus R &\text{ if }m\leq  i\leq 2n+1\\
	 	0&\text{ otherwise.}
	 \end{cases}$$

	Now, we compute the cohomology ring structure of $X$.
	 In the Leray Serre spectral sequence for the Borel fibration $X\stackrel{i}{\hookrightarrow}X_G\stackrel{\pi}{\rightarrow}B_G$, it is easy to observe that  $d_r(1\otimes x)=0=d_r(1\otimes y)$ for all $r\geq 0$, and the images of $u_{2i+1}$ and $v_{m+2i+1}$  must be nonzero  under some differential, for all $0\leq i
     \leq n$. Consequently,   $x^iu_1\not=0$ and $x^iu_1y\not=0$. Obviously, $x^{n+1}=y^2=0$.

	For $m>2n$, it is clear that  for $0\leq i\leq  n$, $u_{2i+1}=\alpha_i x^iu_1$ and $v_{m+2i+1}=\beta_i x^iu_1y$ for some nonzero elements $\alpha_i,\beta_i$ in $R$. In particular, for $m=2n+1$,  $x^nu_1$ can not be equal  to any multiple  of $y$ and so $u_{2n+1}$ is  generated by $\{y,x^nu_1\}$.

	Now, suppose $m\leq 2n$.  If  $m=2i_0+1$ for some $i_0<n$ then  $u_{2i+1}=\alpha_ix^iu_1$  for $0\leq i<i_0$ and $v_{m+2i+1}=\beta_i x^iu_1y$ for $n-i_0\leq i\leq n$, where $\alpha_i$'s and $\beta_i$'s are nonzero elements in 
	$R$. As $x^iu_1$ can not be equal to $x^{i-i_0}y$, for all $i_0\leq i\leq n$, therefore, $u_{2i+1}$ is generated by $\{x^iu_1,x^{i-i_0}y\}$, for all $i_0\leq i\leq n$. Also, $x^iu_1y$ can not be equal to any multiple of $x^{i+i_0+1}$, so $v_{m+2i+1}$ is generated by $\{x^{i+i_0+1},x^iu_1y\}$, for all $0\leq i<  n-i_0$. 
	If $m=2i_0$, for some $i_0\leq n$ then $u_{2i+1}=\alpha_i x^iu_1$ for $0\leq i<i_0$ and  $v_{2i+m+1}=\beta_i x^iu_1y$  for $n-i_0+1\leq i\leq n$, where  $\alpha_i$'s  and $\beta_i$'s are nonzero elements in $R$. We observe that 
	$x^iu_1$ can not be equal to any multiple of $x^{i-i_0}u_1y$,  for all $ i_0\leq i\leq n$. Thus, the elements  $u_{2i+1}$ and $v_{2i+1}$ are generated by $\{ x^iu_1, x^{i-i_0}u_1y \}$.
Note that $u_1^2$ may be both zero or nonzero. If $u_1^2=0$ then $X\sim_R\mathbb{S}^1\times\mathbb{CP}^n\times\mathbb{S}^m$. If $u_1^2\not=0$, then by the commutativity of cup product, $R=\Z_2$ and  $X\sim_{\Z_2} \mathbb{RP}^{2n+1}\times \mathbb{S}^m$. This realizes  possibilities (ii) and (iii) of the theorem.

	\noindent
	\textbf{Case(b):} When $\cup:H^0(X/G)\to H^2(X/G)$ maps $1$ to $ca$ for some $c\not=0$ in $R$. 
	
	First, suppose that $m>2n$.   In this case,   $\rho_{2i+1+j}$ and  $p^*_{2i+2+j}$ are trivial for all $0\leq i<n$ and $j=0$ or $m$. Consequently, $H^{2i+1+j}(X)=H^{2i+2+j}(X)=0$. Note that $H^i(X)=0$ for $2n+1<i<m$ and $H^{m+2n+1}(X)\cong R$ with basis $\{u_{m+2n+1}\}$, where $\rho_{m+2n+1}(u_{m+2n+1})=a^nb$. Thus, for $m\not = 2n+1$,  $\rho_{2n+1}$ and $p^*_m$ are isomorphisms. This implies that $H^{2n+1}(X)\cong H^m(X)\cong R$  with bases $\{u_{2n+1}\}$ and $\{y\}$ respectively. If $m=2n+1$ then $H^{2n+1}(X)\cong R\oplus R$ with basis $\{u_{2n+1},y\}$, where $\rho_{2n+1}(u_{2n+1})=a^n$ and $p^*_m(b)=y$.

	It is easy to see that, for $m\leq 2n$,  the  cohomology groups and generators are the same as above.

	 In  the Leray-Serre spectral sequence for the Borel fibration $X\hookrightarrow X_G\stackrel{\pi}{\to }B_G$, $d_{2r'}(1\otimes u_{2n+1})\not= 0$  for some $r'>0$ and $d_{2r}(1\otimes y)=0$ for all $r\geq 0$. Let if possible  $m=2i_0$   and $r'=n-i_0+1$, where $1\leq i_0\leq n$. Then  as done in previous theorem, we get $H^*(X_G)\cong R[u,v]/\langle u^{n-i_0+1}v-\alpha u^{n+1},v^2-\beta u^{2i_0}-\gamma u^{i_0}v ,u^{n+i_0+1}\rangle$
	 where $\deg u=2,\deg v=m$ and $\alpha,\beta,\gamma\in R$, 
	 $\beta=0$  if  $m>2n$, which is a contradiction.
	  Therefore,  $r'$ must be $n+1$. This implies that  $d_{2n+2}(1\otimes yu_{2n+1})\not = 0$. Consequently, $u_{m+2n+1}=\alpha yu_{2n+1}$ for some $\alpha \not=0$ in $R$.   Obviously,  $y^2=0$ and $u_{2n+1}^2=0$ for $m\not\in \{2n+1,4n+2\}$. If $m=2n+1$ then $u_{2n+1}^2\not =\alpha'u_{m+2n+1}$ for any $\alpha'$ in $R$ and so $u_{2n+1}^2=0$. If $m=4n+2$ then $u_{2n+1}^2$ may be both zero or nonzero. If $u_{2n+1}^2=0$ then
	$H^*(X)=R[y,u_{2n+1}]/\langle y^2,u_{2n+1}^2\rangle$, where $\deg y=m$ and $\deg u_{2n+1}=2n+1$.  This realizes possibility (i)  of the theorem.   
	 If $u_{2n+1}^2\not=0$ then $u_{2n+1}=\beta y$ for some nonzero element $\beta$ in $R$. So, we get $H^*(X)=R[u_{2n+1}]/\langle u_{2n+1}^4\rangle$, where  $\deg u_{2n+1}=2n+1$ and $m=4n+2$. By the commutativity of cup product this is not possible for $R=\mathbb{Q}$. Also, by \cite[Theorem 4L.9]{Hatcher}, it is not possible for $R=\Z_2$.
	
	\noindent\textbf{Case(c):} When $\cup:H^0(X/G)\to H^2(X/G)$ maps $1$ to $cb$, where $c\not=0$ in $R$.

	In this case, $m$ must be 2 and $H^1(X)=0$. 
	As $H^{2i-1}(X/G)=0$ for all $i\geq 0$, therefore, we have  $H^{2i-2}(X)\cong \im p^*_{2i-2}$ and $H^{2i+1}(X)\cong \ker \{\cup:H^{2i}(X/G)\to H^{2i+2}(X/G)\}$ with bases $\{x^i\}$ and $\{u_{2i+1}\}$, respectively, where  $\rho_{2i+1}(u_{2i+1})=a^{i-1}b$ and $p^*_2(a)=x$ for all $1\leq i\leq n+1$. 
	 In the Leray-Serre spectral sequence,  $d_{2r}(1\otimes x)=0$ for all $r\geq 0$ and $d_{2r'}(1\otimes u_3)\not=0$ for some $r'>0$. If $r'=1$ then $H^*(X_G)\cong R[u,v]/\langle u^{n+2},v^{n+1},uv\rangle$, where $\deg u=\deg v=2$, a contradiction. Therefore, $r'$ must be $2$, and we get 
	 $x^iu_3\not=0$ for all $1\leq i\leq n$. This implies that  $u_{2i+3}=\alpha_ix^iu_3$ for some nonzero $\alpha_i\in R$. We have $u_3^2=\alpha x^3$ for some $\alpha\in R$. By the commutativity of cup product, we get $2u_3^2=0$. Therefore, for $R=\mathbb{Q}$, $\alpha$ must be zero.  Therefore,  $H^*(X)\cong R[x,u_3]/\langle x^{n+1},u_3^2-\alpha x^3 \rangle$, where $\deg x=2,$ $\deg u_3=3$ and $\alpha=0$ for $R=\mathbb{Q}$. This realizes possibility (iv).

	\noindent\textbf{Case(d):}  When $\cup:H^0(X/G)\to H^2(X/G)$ maps $1$ to $ca+c'b$, where $c,c'\not=0$ in $R$.

We have $m=2$ and $\cup:H^{2i}(X/G)\to H^{2i+2}(X/G)$ is an isomorphism, for all $0< i<n$. By the exactness of Gysin sequence, $H^{2i-1}(X)=H^{2i+2}(X)=0$ for all $0< i\leq n$; and $H^2(X)\cong H^{2n+1}(X)\cong H^{2n+3}(X)\cong R$ with bases $\{x\}$, $\{u_{2n+1}\}$ and $\{u_{2n+3}\}$, respectively, where  $p^*_2(a)=x$, $\rho_{2n+1}(u_{2n+1})=a^{n-1}b-\frac{c}{c'}a^n$ and $\rho_{2n+3}(u_{2n+3})=a^nb$. Hence,  $H^*(X)\cong R[x,u_{2n+1}]/\langle u_{2n+1}^2, x^{2}\rangle$, where $\deg x=2$ and $\deg u_{2n+1}=2n+1$. This realizes possibility (i). 
\end{proof}
\end{theorem} 	
 Finally, we classify a finitistic space $X$ equipped with a free involution and the orbit space  a product of real projective space and sphere: 
 	\begin{theorem}\label{Theorem S^0}
 	Let $G=\mathbb{Z}_2$ act freely on a finitistic connected space $X$ with $X/G\sim_{\Z_2}  \mathbb{RP}^n\times\mathbb{S}^m $. Then  the cohomology algebra of $X$ is one of the following:
 	\begin{enumerate}
 		\item[(i)] $\S^{m}\times\mathbb{S}^n$;
 		\item[(ii)] $\Z_2[x,y]/\langle x^{n+1},y^2+\alpha x^2 \rangle$, where $\deg x=\deg y=1,m=1$ and $\alpha\in \Z_2$;
 		\item[(iii)] $\Z_2[y]/\langle y^4\rangle,$ where $ \deg y=n$,  $m=2n$ and $n=1,2$ or 4.	
 	\end{enumerate} 
 \end{theorem}
  \begin{proof}
  	Clearly, $H^0(X/G)\cong H^0(X)$ and $H^i(X)=0$ for all $i>m+n$. It is easy to see that Euler class of the 0-sphere bundle $X\to X/G$ must be nontrivial.  Now, we consider the following cases:
 	
 	\noindent
 	\textbf{Case(a):} When $\cup:H^0(X/G)\to H^1(X/G)$ maps $1$ to $a$. 
 	
 	First, we suppose that $m\geq n$.  For $m\not= n$, we get   $\cup:H^{k+i}(X/G)\to H^{k+i+1}(X/G)$ is an isomorphism for all $1\leq i<n$ and $k=0$ or $m$. This implies that $\rho_{k+i}$ and  $p^*_{k+i+1}$ are trivial homomorphism.  Consequently, $H^{k+i}(X)=0$  and $H^{n}(X)\cong H^m(X)\cong R$  with bases $\{u_{n}\}$ and $\{y\}$ respectively, where $\rho_n(u_n)=a^n$ and $p^*_m(b)=y$. Note that $H^i(X)=0$ for $n<i<m$, $H^{m+n}(X)\cong R$ with basis $\{u_{m+n}\}$, where $\rho_{m+n}(u_{m+n})=a^nb$.
 	 For $n=m$,  $H^n(X)\cong R\oplus R$ with basis $\{u_{n},y\}$,  where $\rho_n(u_n)=a^n$ and $p^*_m(b)=y$.

 	Next, suppose that $m< n$. We get $\cup:H^i(X/G)\to H^{i+1}(X/G)$ is an isomorphism for all $1<i<n+m$ and $i\not= m-1,n$. Note that  the map $\cup:H^{m-1}(X/G)\to H^{m}(X/G)$    is  injective  with  $\im\cong R$ with basis $\{a^m\}$, and  for $i=n$ it is surjective with $\ker\cong R$ with basis $\{a_n\}$, where $\rho_n(u_n)=a^n$. This implies that $H^i(X)=0$ for  $1<i<n+m$ and $i\not= m,n$, and $H^m(X)\cong H^n(X)\cong  R$ with basis $\{y\}$ and $\{u_n\}$, respectively,  where $p^*_m(b)=y$ and $\rho_n(u_n)=a^n$. Also, $H^{n+m}\cong R$ with basis $\{u_{n+m}\}$, where $\rho_{n+m}(u_{n+m})=a^nb $.

 	Now, we compute the cohomology ring structure of $X$. In  the Leray-Serre spectral sequence for the Borel fibration $X\hookrightarrow X_G\stackrel{\pi}{\to }B_G$, $d_{r'}(1\otimes u_{n})\not= 0$  for some $r'>0$ and $d_{r}(1\otimes y)=0$ for all $r\geq 0$. Let it possible for $n>m$,  $r'=n-m+1$. Then  we get $H^*(X_G)\cong R[u,v]/\langle u^{n-m+1}v-\alpha u^{n+1},v^2-\beta u^{2m}-\gamma u^{m}v ,u^{n+i_0+1}\rangle$
 	where $\deg u=1,\deg v=m$ and $\alpha,\beta,\gamma\in R$, 
 	$\beta=0$  if  $m>n$, which is a contradiction. Therefore,  $r'$ must be $n+1$. This implies that  $d_{n+1}(1\otimes yu_{n})\not = 0$. Consequently, $u_{m+n}=yu_{n}$.   Obviously,  $y^2=0$, and $u_{n}^2=0$ for $m\not\in \{n,2n\}$. If $m=n$ then $u_{n}^2\not =u_{m+n}$ and so $u_{n}^2=0$. If $m=2n$ then $u_{n}^2$ may be both zero or nonzero. If $u_n^2=0$ then  	$H^*(X)=\Z_2[y,u_{n}]/\langle y^2,u_{n}^2\rangle$, where $\deg y=m$ and $\deg u_{n}=n$. If $u_n^2\not=0$  then the cohomology algebra of $X$ is $\Z_2[u_n]/\langle u_{n}^4\rangle$, where  $\deg u_{n}=n$ and $m=2n$. By \cite[Theorem 4L.9]{Hatcher}, $n=1,2,4$ or 8.  This realizes possibility (i) and (iii).

 	\noindent\textbf{Case(b):} When $\cup:H^0(X/G)\to H^2(X/G)$ maps $1$ to $b$.

 	In this case, $m$ must be 1 and  $\im p^*_1=\Z_2$ with basis $\{p^*_1(a)\}$. Also,   $\im p^*_{i}\cong \im \rho_i\cong \Z_2$ with basis $\{p^*_i(a^{i})\}$  and $\{a^{i-1}b\},$  respectively,   for all $0<i\leq n$. Consequently, $H^{i}(X)\cong \Z_2\oplus \Z_2$ with basis $\{x^i,u_{i}\}$, where $\rho_{i}(u_{i})=a^{i-1}b$ and $p^*_1(a)=x$. Obviously, $H^{n+1}(X)\cong \Z_2$ with  basis $\{u_{n+1}\}$,  where $\rho_{n+1}(u_{n+1})=a^nb$. In the Leray-Serre spectral sequence,  $d_{r}(1\otimes x)=0$ for all $r\geq 0$ and $d_2(1\otimes u_1)\not=0$. This implies that  for $1\leq i\leq n$,
 	$x^iu_1\not=0$, and hence $u_{i+1}=x^iu_1+\alpha_i x^{i+1}$ for some $\alpha_i\in\Z_2$. As $d_2(1\otimes xu_1)\not=0$, we have $u_1^2=\alpha x^2$ for some $\alpha\in\Z_2$. Therefore, $H^*(X)\cong \Z_2[x,u_1]/\langle x^{n+1},u_1^2+\alpha x^2 \rangle$, where $\deg x=\deg u_1=1$. This realizes possibility (ii).
 	
 	\noindent\textbf{Case(c):}  When $\cup:H^0(X/G)\to H^2(X/G)$ maps $1$ to $a+b$.

 	In this case also $m$ must be 1 and $\cup:H^{i}(X/G)\to H^{i+1}(X/G)$ is an isomorphism for all $0< i\leq n$. Consequently, $H^{i}(X)=0$ for $1<i<n$. Note that $\ker p^*_1\cong\im \rho_n\cong \Z_2$ with bases $\{a+b\}$ and $\{a^n+a^{n-1}b\}$, respectively. This implies that for $n\not=1$, $H^1(X)\cong H^n(X)\cong \Z_2$ with bases $\{x\}$ and $\{u_n\}$, respectively,  where $p^*_1(a)=x$ and $\rho_n(u_n)=a^n+a^{n-1}b$. Obviously,  $H^{n+1}(X)\cong\Z_2$ with basis $\{u_{n+1}\}$, where $\rho_{n+1}(u_{n+1})=a^nb$. In particular for $n=1$, $H^1(X)\cong\Z_2\oplus\Z_2$, $H^2(X)\cong \Z_2$ with bases $\{x,u_1\}$ and $\{u_2\}$, respectively, where $p^*_1(a)=x,\rho_1(u_1)=a+b$ and $\rho_2(u_2)=ab $. We must have $u_{n+1}=xu_n$ and $u_n^2=0$. Thus, $X\sim_{\Z_2}\mathbb{S}^1\times\mathbb{S}^n$. 	 This realizes possibility (i).
 \end{proof}                          
Now, we determine covering dimension   of coincidence set $A(f)$  of continuous maps $f:X\to \mathbb{R}^k$, where $X$ is a finitistic space equipped with free involution and $X/G\sim_{\Z_2}\mathbb{RP}^n\times\mathbb{S}^m $.
\begin{theorem}
	Let $G=\Z_2$ act freely on a finitistic space $X$ with $X/G\sim_{\Z_2}\mathbb{RP}^n\times\mathbb{S}^m,$ where $ n>4,m>1  $. If $f:X\to \mathbb{R}^k$ be any continuous map, then  $cov.dim(A(f))\geq n-k$ for $k\leq n$.
\end{theorem}
\begin{proof}
By  Theorem \ref{Theorem S^0},  the Volovikov's index   $in(X)$ \cite{Volovikov1992} is  $n$. Note that $in(A(f))\geq  in(X)-k$ \cite{Volovikov2000}, so  we have $in(A(f))\geq n-k$. As $G$ acts freely on $X$,  it induces a free action on $A(f)$. By the  definition of $in(X)$, $H^{n-k}(B_G)\to H^{n-k}(A(f)_G) $ is injective. This implies that $cohom.dim(A(f)_G)\geq n-k$, where $cohom.dim$ denotes the cohomological dimension of a space.  As $A(f)/G$ and $A(f)_G$  are homotopy equivalent,  $cohom.dim(A(f)/G)\geq n-k$. Consequently,  $cohom.dim(A(f))\geq n-k$ \cite[Proposition A.11]{Quillen }. The result follows from the fact that $cov.dim(A(f))\geq cohom.dim(A(f))$.  
\end{proof}

\section{Examples}                                 
Consider the standard  free actions  of $G=\mathbb{S}^{d}$ on $\mathbb{S}^{(d+1)n+d}$, for  $d=0,1$ or 3, and trivial action on $\mathbb{S}^m$, then under the diagonal action $G$ acts freely on  $\mathbb{S}^{(d+1)n+d}\times \mathbb{S}^m$ with the  orbit space  $\mathbb{FP}^n\times \mathbb{S}^m$, where $\mathbb{F}=\mathbb{R},\mathbb{C}$ or $\mathbb{H}$, respectively. This realizes the possibility (i) of Theorems \ref{Thoerem S3}, \ref{Theorem for S^1} and \ref{Theorem S^0}. Similarly, if we take  free actions of $G=\mathbb{S}^d$ on $\mathbb{S}^{2d+1}$  and trivial action on $\mathbb{FP}^n$ then  $G$ acts freely on  $\mathbb{FP}^n\times\mathbb{S}^{2d+1} $ with the orbit space $\mathbb{FP}^n\times\mathbb{FP}^1 $, respectively. This realizes  the possibility (iii) of Theorem \ref{Thoerem S3},  when $\alpha=0$, the possibility (iv) of Theorem \ref{Theorem for S^1} and the possibility (ii) of Theorem \ref{Theorem S^0}.  If we   take free action of  $G=\mathbb{S}^{d}$ on itself, for $d=1$ or 3, and trivial action on  $ \mathbb{FP}^n\times\mathbb{S}^m$, where $\mathbb{F}=\mathbb{C}$ or $\mathbb{H}$, respectively then $G$ acts freely on  $\mathbb{S}^{d}\times \mathbb{FP}^n\times\mathbb{S}^m$ with the orbit space  $\mathbb{FP}^n\times\mathbb{S}^m$. This realizes the possibility (ii) of Theorems  \ref{Thoerem S3} and \ref{Theorem for S^1}.

Now, consider free action of $\Z_4$ on $\mathbb{S}^3\subseteq \mathbb{C}^2$  defined by $(z_1,z_2)\mapsto (z_1 e^{2\pi i/4},z_2 e^{2\pi i/4})$. This induces a free involution   on $\mathbb{RP}^3$ with the orbit space $\text{L}^3(4,1)$. We know that $\text{L}^3(4,1)\sim_{\Z_2}\mathbb{RP}^1\times  \mathbb{S}^2 $. Recall that if  $G=\Z_2$ acts freely on a finitistic space $X$ with the mod 2 cohomology $\mathbb{CP}^3$, then the orbit spaces $X/G$ is the mod 2 cohomology  $\mathbb{RP}^2\times \mathbb{S}^4$ \cite{Hemant2008}. These examples realizes the possibility (iii) of Theorem \ref{Theorem S^0} for $n=1$ and $n=2$, respectively.
\vspace{0.35cm}

\noindent\textbf{Remark.} Let $X$ be a connected Hausdorff space equipped with free action of  $G=\Z_2$. If the orbit space $X/G=\mathbb{RP}^n\times \mathbb{S}^m$ then $X$ is homeomorphic to $\mathbb{S}^n\times \mathbb{S}^m$.  The significance of Theorem \ref{Theorem S^0}  lies with the fact that if $X/G\sim_{\Z_2}\mathbb{RP}^n\times \mathbb{S}^m$ then $X$ may have the mod 2 cohomology  isomorphic to    $\mathbb{RP}^n\times \mathbb{S}^1 $ or $\mathbb{S}^n\times\mathbb{S}^1$ for $m=1$; and $\mathbb{S}^4\times \mathbb{S}^2$ or $\mathbb{CP}^3$ for $m=4$ and $n=2$.

  \bibliographystyle{plain}

\end{document}